\documentclass{amsart}

\usepackage[english]{babel}
\usepackage{csquotes} 
\usepackage{graphicx} 
\usepackage{amsmath, amsfonts, amssymb, amsthm}
\usepackage{mathrsfs} 
\usepackage{quiver} 

\usepackage{enumitem}
\setlist[itemize]{noitemsep, topsep=-3pt}
\setlist[enumerate]{noitemsep, topsep=-3pt}

\usepackage[
backend=biber,
style=alphabetic,
sorting=nyt,
maxbibnames=99
]{biblatex}

\title{When is the matroid Schubert variety $\mathbb{Q}$-Gorenstein?}
\author{Townsend Porcher}

\theoremstyle{plain}
\newtheorem{theorem}{Theorem}[section]
\newtheorem{corollary}[theorem]{Corollary}
\newtheorem{lemma}[theorem]{Lemma}
\newtheorem{proposition}[theorem]{Proposition}
\theoremstyle{definition}
\newtheorem{definition}[theorem]{Definition}
\newtheorem{example}[theorem]{Example}
\newtheorem{remark}[theorem]{Remark}

\newtheorem{opchowring}[theorem]{Operational Chow Ring}

\newcommand{\rk}{\operatorname{rk}}

\newcommand{\cl}{\operatorname{cl}}
\newcommand{\Cl}{\operatorname{Cl}}
\newcommand{\CaCl}{\operatorname{CaCl}}
\newcommand{\Pic}{\operatorname{Pic}}
\newcommand{\proj}{\operatorname{proj}}
\newcommand{\id}{\operatorname{id}}

\newcommand{\Span}{\operatorname{span}}
\newcommand{\MW}{\operatorname{MW}}

\makeatletter
\newcommand*{\textoverline}[1]{$\overline{\hbox{#1}}\m@th$}
\makeatother

\begin{document}

\begin{abstract}
Let $E$ be a finite set, and let $V\subseteq \mathbb{C}^E$ be a linear subspace that is not contained in any coordinate hyperplane. The closure of $V$ in the product of projective lines $(\mathbb{P}^1)^E$ is a singular variety $Y_V$ known as the matroid Schubert variety (or arrangement Schubert variety). We use operational Chow cohomology to prove that every line bundle on $Y_V$ is the restriction of a line bundle on $(\mathbb{P}^1)^E.$ We then give combinatorial characterizations of when $Y_V$ is Gorenstein and $\mathbb{Q}$-Gorenstein, respectively. We provide examples of linear subspaces $V$ such that $Y_V$ is Gorenstein but not smooth, $\mathbb{Q}$-Gorenstein but not Gorenstein, and not $\mathbb{Q}$-Gorenstein, respectively.
\end{abstract}

\maketitle

\section{Introduction}

Let $E$ be a finite set, and let $V \subseteq \mathbb{C}^E$ be an $r$-dimensional linear subspace such that $V$ is not contained in any coordinate hyperplane of $\mathbb{C}^E$. The embedding of $\mathbb{C}$ into $\mathbb{P}^1$ by adding a point at infinity induces an embedding of $\mathbb{C}^E$ into $(\mathbb{P}^1)^E$. The closure of $V$ in $(\mathbb{P}^1)^E$, denoted by $Y_V$, is called the \textit{matroid Schubert variety} of $V$. The matroid Schubert variety was introduced in \cite{ardila_boocher}. It is not a Schubert variety in the classical sense, but bears a similar name due to its geometric similarities (cf. \cite[Section 1]{singular_hodge_theory}). Studying the geometry of the matroid Schubert variety finds applications in matroid theory; notably, in \cite{huh_wang}, it plays a central role in resolving the realizable case of the Dowling-Wilson conjecture.

Many geometric invariants of the matroid Schubert variety $Y_V$ are actually combinatorial invariants of the matroid of the linear subspace $V$ (cf. \cite[Theorem 1.1]{ardila_boocher}). The \textit{matroid} of $V,$ denoted by $M_V$, is the rank-$r$ matroid on $E$ whose collection of independent sets is \[\{S \subseteq E : \text{the linear map }V \hookrightarrow \mathbb{C}^E \xrightarrow{\proj_S} \mathbb{C}^S \text{ is surjective}\}.\] 
The condition that $V$ is not contained in any coordinate hyperplane of $\mathbb{C}^E$ is equivalent to requiring that $M_V$ is a loopless matroid.

The matroid Schubert variety is singular except when $V = \mathbb{C}^E$  \cite[Corollary 3.9]{singular_thesis}. In this paper, we give combinatorial characterizations of when the matroid Schubert variety is Gorenstein and $\mathbb{Q}$-Gorenstein, respectively. The following theorem is proven in Section 4.

\begin{theorem}
\label{thm:gorenstein_condition}
    The matroid Schubert variety $Y_V$ is Gorenstein (resp. $\mathbb{Q}$-Gorenstein) if and only if there exists
    a weight function $w:E\to \mathbb{Z}$ (resp. $w:E\to \mathbb{Q}$) that satisfies
    \[\sum_{e \in C} w(e) = -2\]
    for all cocircuits $C$ of $M_V$.
\end{theorem}

The outline of the paper is as follows: Section 2 contains matroid-theoretic notations, definitions, and lemmas that are needed for the rest of the paper. In Section 3, we prove Proposition \ref{prop:line_bundles}, which states that every line bundle on $Y_V$ is the restriction of a line bundle on $(\mathbb{P}^1)^E.$ This proposition is used in the proof of Theorem \ref{thm:gorenstein_condition} and may be of independent interest. In Section 4, we prove Theorem \ref{thm:gorenstein_condition}. In Appendix A, we prove a technical lemma that is used for the proof of Proposition \ref{prop:line_bundles}. We end this section with some concrete examples.

\begin{example}
Let $v_1,v_2,v_3$ be the standard basis vectors of $\mathbb{C}^3$, and let 
\[V = \Span(v_1 + v_3, v_2+v_3) \subseteq \mathbb{C}^3.\]
Then, the matroid of $V$ is the uniform matroid $U_{2,3}$ on the set 123, whose cocircuits are 12, 13, and 23. It is straightforward to check that there is exactly one weight function $w: 123 \to \mathbb{Z}$ such that the weight of each cocircuit of $U_{2,3}$ is -2: it is the constant function -1. Therefore, by Theorem \ref{thm:gorenstein_condition}, the closure of $V$ in $(\mathbb{P}^1)^3$ is Gorenstein, but not smooth.
\end{example}

\begin{example}
\label{cool_example}
Let $v_1,v_2,\dots v_7$ be the standard basis vectors of $\mathbb{C}^7,$ and let 
\begin{equation*}
\begin{split}
V = \Span(& v_1+v_5+v_6+v_7, \\ & v_2+v_4+v_6+v_7, \\ & v_3 + v_4 + v_5 + v_7) \subseteq \mathbb{C}^7.
\end{split}    
\end{equation*}
The matroid of $V$ is the rank-3 matroid on the set $1234567$ whose non-bases are 126, 135, 147, 234, 257, and 367, a.k.a. the non-Fano matroid $F_7^-$. One can check that there is exactly one weight function $w: 1234567 \to \mathbb{Q}$ such that the weight of each cocircuit of $F^-_7$ is -2: the function
\[1,2,3,7 \mapsto -\frac{1}{3} \quad \text{and} \quad 4,5,6 \mapsto -\frac{2}{3}.\]
Thus, by Theorem \ref{thm:gorenstein_condition}, the closure of $V$ in $(\mathbb{P}^1)^7$ is $\mathbb{Q}$-Gorenstein, but not Gorenstein.
\end{example}

\begin{example}
Let $v_1,v_2,\dots v_6$ be the standard basis vectors of $\mathbb{C}^6,$ and let 
\begin{equation*}
\begin{split}
V = \Span(& v_1+v_4+v_5, \\ & v_2+v_4+v_6, \\ & v_3 + v_5 - v_6) \subseteq \mathbb{C}^6.
\end{split}    
\end{equation*}
The matroid of $V$ is the rank-3 matroid on the set $123456$ whose non-bases are 124, 135, 236, and 456, a.k.a. the matroid $M(K_4)$ of the complete graph on four vertices. One can check that there does not exist a weight function $w: 123456 \to \mathbb{Q}$ such that the weight of each cocircuit of $M(K_4)$ is -2.
Thus, by Theorem \ref{thm:gorenstein_condition}, the closure of $V$ in $(\mathbb{P}^1)^6$ is not $\mathbb{Q}$-Gorenstein.
\end{example}

\begin{remark}
Example \ref{cool_example} shows us that, for some choices of $V$, the matroid Schubert variety $Y_V$ can be $\mathbb{Q}$-Gorenstein, but not Gorenstein. In this sense, matroid Schubert varieties differ from classical Schubert varieties, which are $\mathbb{Q}$-Gorenstein if and only if they are Gorenstein \cite[Section 3]{woo_yong}.
\end{remark}

\subsection*{Acknowledgments} I thank my advisor, Jeremy Usatine, for his invaluable mentorship and guidance. I thank June Huh for introducing me to the matroid Schubert variety. The title of this paper is an homage to \cite{woo_yong}. This work was partially supported by Simons Foundation MPS-TSM-00007918 and NSF DMS-2502347.

\section{Requisite matroid theory}
In this section, we:
\begin{enumerate}
    \item[(1)] establish some notation,
    \item[(2)] define ``compatible pairs of a matroid,'' which will be used in Section \ref{sec:line_bundles},
    \item[(3)] collect an assortment of matroid-theoretic lemmas that we need.
\end{enumerate}

Let $M$ be a rank-$r$ matroid on a finite set E. We write:
\begin{itemize}
    \item $\mathcal{B}(M)$ to denote to the collection of bases of $M$,
    \item $\mathcal{L}(M)$ to denote the collection of flats of $M$,
    \item $\mathcal{L}_k(M)$ to denote the collection of rank-$k$ flats of $M$,
    \item $\mathcal{L}_{\neq r}(M)$ to denote the collection of proper flats of $M$,
    \item $\mathcal{H}(M)$ to denote the collection of corank-1 flats of $M$,
    \item $\mathcal{H}_e(M)$ to denote the collection of corank-1 flats of $M$ that do not contain $e \in E$,
    \item $M | S$ to denote the matroid $M$ restricted to the subset $S \subseteq E.$
\end{itemize}

\begin{definition}
A \textit{compatible pair} of $M$ is an ordered pair $(I, \mathscr{F}),$ where $I$ is an independent set of $M$, and $\mathscr{F}$ is a flag of proper flats of $M,$ such that $I$ is contained in every flat of $\mathscr{F}.$ The set $I$ and/or the flag $\mathscr{F}$ may be empty.
\end{definition}

Compatible pairs of $M$ enumerate the cones of the augmented Bergman fan of \cite{semismall}. ``Compatible pairs'' in the sense of \cite{semismall} are, under our definition, compatible pairs of the Boolean matroid.

\begin{definition}
A compatible pair $(I, \mathscr{F})$ of $M$ is \textit{maximal} if $\rk(I) + |\mathscr{F}| = r.$ A compatible pair is \textit{nearly maximal} if $\rk(I) + |\mathscr{F}| = r - 1$ and $\cl(I) \notin \mathscr{F}.$
\end{definition}

Maximal compatible pairs correspond to maximal cones of the augmented Bergman fan. Nearly maximal compatible pairs play a key role in Section \ref{sec:line_bundles}.

\begin{lemma}
\label{lem:flat_disjoint_union}
If $M$ is loopless, then the rank-1 flats of $M$ are pairwise disjoint, and each flat of $M$ can be written as the disjoint union of rank-1 flats.
\end{lemma}
\begin{proof}
The intersection of two distinct rank-1 flats of $M$ is the rank-0 flat $\cl(\emptyset).$ Since $M$ is loopless, $\cl(\emptyset) = \emptyset,$ so we have pairwise disjoint-ness.

Let a flat $F$ of $M$ be given. We first note that
\[\bigcup_{e \in F} \cl(e) \supseteq \bigcup_{e \in F} e = F.\]
We then note that, for every $e \in F,$ we have $\cl(e) \subseteq F.$ Since $M$ is loopless, the flat $\cl(e)$ has rank 1 for all $e \in F,$ so we are done.
\end{proof}

\begin{lemma}
\label{lem:matroid_pairwise_disjoint}
Let $M | S$ denote the matroid $M$ restricted to the subset $S \subseteq E.$ The set $\mathcal{B}(M|F)$ is non-empty for all flats $F$ of $M$.
If $F_1$ and $F_2$ are distinct flats of $M$, then
\[\mathcal{B}(M | F_1) \, \cap \, \mathcal{B}(M | F_2) = \emptyset.\]
\end{lemma}

\begin{proof}
The set of bases of a matroid is always non-empty. For the pairwise disjoint-ness, suppose that $B\in \mathcal{B}(M_V | F) \cap \mathcal{B}(M_V | F') $. Then, $F = \cl(B) = F'.$
\end{proof}

\begin{lemma}
\label{lem:matroid_set_equality}
For any $e \in E,$ let $\mathcal{B}_e(M)$ denote the set
\[\{B \setminus e \, : \, B \in \mathcal{B}(M), \: e\in B\}.\]
For all elements $e \in E$, we have the equality of sets.
\begin{equation*}
\label{eq:matroid_set_equality}
\bigsqcup_{F \in \mathcal{H}_e(M)} \mathcal{B}(M | F) = \mathcal{B}_e(M).
\end{equation*}
\end{lemma}

\begin{proof}

    If $e$ is a loop, then both sides of the equality are the empty set. So we assume that $e$ is not a loop.

    To show that the left-hand side is contained in the right-hand side, let $F \in \mathcal{H}_e(M)$ be given, and let $B$ be a basis of $M | F$. It will suffice to show that $B\cup e$ is a basis of $M$. Since $F$ is a flat of corank 1, the cardinality of the set $B$ is $r-1$, so the cardinality of the set $B \cup e$ is $r$. Additionally, $F$ must be a proper subset of $\cl(B \cup e),$ since $e \notin F$. The flat $F$ is corank-1, so the flat $\cl(B \cup e)$ must be corank-0; hence, the rank of $B \cup e$ is $r$. In conclusion, $B\cup e$ is a cardinality-$r$, rank-$r$ set; therefore, it is a basis of $M$.

    To show that the right-hand side is contained in the left-hand side, let $B$ be a basis of 
    $M$ containing $e$. Let $F = \cl(B \setminus e).$ Since $B \setminus e$ is an rank-$(r-1)$ independent set, we have that $F$ is a rank-$(r-1)$ flat, and that $B \setminus e \in \mathcal{B}(M|F)$. Lastly, it cannot be that $e \in F$, as this would imply that $B \subseteq F$, and a rank-$r$ independent set cannot be a subset of a rank-$(r-1)$ flat.
\end{proof}

\section{Line bundles on the matroid Schubert variety}
\label{sec:line_bundles}

The objective of this section is to prove the following proposition.

\begin{proposition}
\label{prop:line_bundles}
    Every line bundle on $Y_V$ is the restriction of a line bundle on $Y_{\mathbb{C}^E}$.
\end{proposition}

Let us consider the lattice $\mathbb{Z}^E$ with standard basis elements $\{v_e\}_{e\in E}$. For any subset $S \subseteq E$, let
\[v_S :=\sum_{e\in S}v_e.\]
We describe two complete, unimodular fans on $\mathbb{Z}^E$: $\Gamma_E$ and $\Delta_{E}$.

The cones of $\Gamma_{E}$ are the Minkowski sums
\[\tilde{\sigma}_{S_1,S_2} := \text{cone}(v_e : e \in S_1) + \text{cone}(-v_e : e \in S_2)\]
over all pairs of disjoint subsets $S_1,S_2 \subseteq E.$ The set $S_1$ and/or the set $S_2$ may be empty. The toric variety $X(\Gamma_E)$ corresponding to $\Gamma_E$ is the product of projective lines $(\mathbb{P}^1)^E$.

The cones of $\Delta_{E}$ are the Minkowski sums
\[\overline{\sigma}_{I,\mathscr{F}} =:\text{cone}(v_e : e \in I) + \text{cone}\left(-v_{E \setminus F} : F \in \mathscr{F}\right)\]
for each compatible pair $(I, \mathscr{F})$ of the Boolean matroid on $E$. The toric variety $X(\Delta_E)$ associated to $\Delta_E$ is the stellahedral variety of \cite{stellahedral}.

The identity map of lattices $\id:\mathbb{Z}^E \to \mathbb{Z}^E$ induces a proper morphism of toric varieties $\rho:X(\Delta_E) \to (\mathbb{P}^1)^E$. The positive orthant of $\mathbb{Z}^E$ is a cone in both $\Gamma_E$ and $\Delta_{E}$; this gives us a canonical embedding of the corresponding affine toric variety $\mathbb{C}^E$ in both $X(\Delta_E)$ and $(\mathbb{P}^1)^E.$ The closure of the linear subspace $V \subseteq \mathbb{C}^E$ in $X(\Delta_E)$ is the augmented wonderful variety of \cite{semismall}, and is denoted by $W_V$. Let $\kappa$ denote the closed immersion $W_V \subseteq X(\Delta_E).$ The closure of $V$ in $(\mathbb{P}^1)^E$ is the matroid Schubert variety $Y_V$. Let $\iota$ denote the closed immersion $Y_V \subseteq (\mathbb{P}^1)^E.$ The image of the morphism $\rho \circ \kappa$ is $Y_V$, and the induced morphism $\pi: W_V \to Y_V$ is a resolution of singularities.

If $V = \mathbb{C}^E$, then $M_V$ is the Boolean matroid on $E$, $Y_V$ is $(\mathbb{P}^1)^E,$ and $W_V$ is $X(\Delta_E).$ For this reason, we will henceforth write the Boolean matroid on $E$ as $M_{\mathbb{C}^E},$ $(\mathbb{P}^1)^E$ as $Y_{\mathbb{C}^E},$ and $X(\Delta_E)$ as $W_\mathbb{C^E}.$ 

Diagram (\ref{dia:commutative_square}) summarizes the exposition above, and plays a central role in this section.

\begin{equation}
\label{dia:commutative_square}
\begin{tikzcd}
	{W_{\mathbb{C}^E}} && {Y_{\mathbb{C}^E}} \\
	\\
	{W_V} && {Y_V}
	\arrow["{\rho}"', from=1-1, to=1-3]
	\arrow["{\kappa}"', hook, from=3-1, to=1-1]
	\arrow["{\pi}", from=3-1, to=3-3]
	\arrow["{\iota}", hook, from=3-3, to=1-3]
\end{tikzcd}
\end{equation}

In this paper, we use Chow homology and operational Chow cohomology, the latter of which was introduced in \cite{operational_chow_cohomology}. For a variety $X,$ we write $A_*(X)$ to denote its Chow group and $A^*(X)$ to denote its operational Chow ring. If $X$ is a $d$-dimensional smooth variety, then there is a ``Poincar\'e dualilty'' isomorphism from $A^k(X)$ to $A_{d-k}(X)$ given by sending a cohomology class $\alpha \in A^k(X)$ to $\alpha \cap [X] \in A_{d-k}(X) $. All future mentions of Poincar\'e duals or Poincar\'e duality will be in the sense described here. If $X$ is proper, it has a degree map $\deg_X: A_0(X) \to \mathbb{Z}.$

For a subset $S \subseteq E,$ let us denote the corresponding affine stratum of $Y_{\mathbb{C}^E}$ by
\[Z_{\mathbb{C}^E, S} := \ \prod_{e \in S} \mathbb{C} \ \times \prod_{e \in E \setminus S} \{\infty\} \ \subset (\mathbb{P}^1)^E.\]

\begin{proposition}
\label{lem:stratification}
\cite[Proposition 4.10]{proudfoot_2018}
The matroid Schubert variety $Y_V$ can be expressed as the disjoint union of schemes
\[\bigsqcup_{F \in \mathcal{L}(M_V)} Z_{V,F},\]
where $Z_{V,F} = Y_V \cap Z_{\mathbb{C}^E,F} $ for each flat $F$ of $M_V$. The stratum $Z_{V,F}$ is isomorphic to $\mathbb{C}^{\rk(F)},$ and its closure is
\[\overline{Z}_{V,F} = \bigcup_{G \subseteq F} Z_{V,G},\]
where the union is over all flats $G$ of $M_V$ such that $G \subseteq F.$
\end{proposition}

\begin{lemma}
\label{lem:chow_group}
We have
\[A_k(Y_V) = \bigoplus_{F \in \mathcal{L}_k(M_V)}\mathbb{Z} \cdot [\overline{Z}_{V,F}],\]
where $[\overline{Z}_{V,F}]$ denotes the Chow homology class of the closed subvariety $\overline{Z}_{V,F}.$
\end{lemma}
\begin{proof}
This is a direct application of \cite[Corollary to Theorem 1]{llompart_descamps}.
\end{proof}

\begin{lemma}
\label{lem:pushforward_descr}
\cite[Lemma 9.3]{stellahedral}
The pushforward map $\iota_*: A_*(Y_V) \to A_*(Y_{\mathbb{C}^E})$ sends $[\overline{Z}_{V, F}]$ to
\[\sum_{B \in \mathcal{B}(M_V | F)} [\overline{Z}_{\mathbb{C}^E,\, B}].\]
\end{lemma}

\begin{corollary}
\label{cor:pushforward_injective}
The push-forward map $\iota_*: A_*(Y_V) \to A_*(Y_{\mathbb{C}^E})$ of Chow groups is injective.
\end{corollary}
\begin{proof}
By Lemma \ref{lem:pushforward_descr}, the map $\iota_*$ sends the free generator $[\overline{Z}_{V,F}]$ to the sum of free generators
\[\sum_{B \in \mathcal{B}(M_V | F)} [\overline{Z}_{\mathbb{C}^E,\, B}].\]
But by Lemma \ref{lem:matroid_pairwise_disjoint}, the sets $\mathcal{B}(M_V | F)$ are non-empty and pairwise disjoint.
\end{proof}

\begin{opchowring}
The product of projective lines $Y_{\mathbb{C}^E}$ is a smooth, proper toric variety, so \cite[Theorem 10.8]{danilov} gives us a presentation of its operational Chow ring:
\begin{equation*}
\frac{\mathbb{Z}[\tilde{y}_e \mid e \in E]}{<{{\tilde{y}_e}^2}, \, e\in E >}.
\end{equation*}
For each subset $S = \{s_1, s_2,\dots,s_m\}$ of $E,$ we use the symbol $\tilde{y}_S$ to denote the element
\[\tilde{y}_{s_1}\tilde{y}_{s_2}\dots\tilde{y}_{s_m}.\]
For a pair of disjoint subsets $S_1, S_2 \subseteq S$ such that $S_1 \cup S_2 = S,$
we have that $\tilde{y}_S$ is Poincar\'e dual to the Chow homology class $[V(\tilde{\sigma}_{S_1, S_2})],$ i.e., $\tilde{y}_S \, \cap \, [Y_{\mathbb{C}^E}] = [V(\tilde{\sigma}_{S_1, S_2})]$, where $V(\tilde{\sigma}_{S_1, S_2})$ denotes the torus orbit closure corresponding to the cone $\tilde{\sigma}_{S_1, S_2} \in \Gamma_E.$ We note here that $[V(\tilde{\sigma}_{S_1, S_2})] =[\overline{Z}_{\mathbb{C}^E, E \setminus (S_1 \cup S_2)}].$    
\end{opchowring}

\begin{opchowring}
The stellahedral variety $W_{\mathbb{C}^E}$ is also a smooth, proper toric variety, so \cite[Theorem 10.8]{danilov}, combined with \cite[Corollary 3.14]{stellahedral}, gives us a presentation of its operational Chow ring:
\begin{equation*}
\frac{\mathbb{Z}[\bar{x}_S \mid S \subsetneq E] \otimes \mathbb{Z} [\bar{y}_e \mid e \in E]    }{\overline{I}_E},
\end{equation*}
where $\overline{I}_E$ is the ideal generated by
\begin{equation}
\label{eq:stella_ideal_description_1}
\tag{\textoverline{R1}}
\overline{y}_e - \sum_{S\not\ni e}\overline{x}_S \quad \text{ for every element } e \in E,
\end{equation}
\begin{equation}
\label{eq:stella_ideal_description_2}
\tag{\textoverline{R2}}
\overline{x}_{S_1}\overline{x}_{S_2} \quad \text{ for every pair of incomparable proper subsets } S_1, S_2, \text{ and}
\end{equation}
\begin{equation}
\label{eq:stella_ideal_description_3}
\tag{\textoverline{R3}}
\overline{y}_e\overline{x}_S \quad \text{ for every $e \in E$ and for every proper subset $S$ not containing $e$.}
\end{equation}
Let $I = \{i_1,\dots,i_s\}$ be a (possibly empty) $s$-element subset $S\subseteq E$, and let 
\[\mathscr{F}:F_{1} \subsetneq F_{2} \subsetneq \cdots \subsetneq F_{t}\]
be a (possibly empty) increasing sequence of proper subsets of $E$, such that $(I, \mathscr{F})$ is a compatible pair of $M_{\mathbb{C}^E}.$ We use the symbol $\overline{x}_{I, \mathscr{F}}$ to denote the element
\[\overline{y}_{i_1}\overline{y}_{i_2}\cdots \overline{y}_{i_s} \overline{x}_{F_1}\overline{x}_{F_2}\cdots \overline{x}_{F_t} \in A^{s+t}(W_V),\]
which is Poincar\'e dual to the Chow homology class $[V(\sigma_{I, \mathscr{F}})].$ The pullback map ${\rho}^*: A^*(Y_{\mathbb{C}^E}) \to A^*(W_{\mathbb{C}^E})$ sends $\tilde{y}_e$ to $\overline{y}_e$.
\end{opchowring}

\begin{opchowring}
We have from \cite[Remark 2.16]{semismall} 
that the operational Chow ring of $W_V$ is
\begin{equation*}    
\frac{\mathbb{Z}[x_F \mid F \in \mathcal{L}_{\neq r} (M_V)] \otimes \mathbb{Z} [y_e \mid e \in E]}{I_M}
\end{equation*}
where $I_M$ is the ideal generated by
\begin{equation}
\label{eq:ideal_description_1}
\tag{R1}
y_e - \sum_{F\not\ni e} x_F \quad \text{ for every element } e \in E,
\end{equation}
\begin{equation}
\label{eq:ideal_description_2}
\tag{R2}
x_{F_1}x_{F_2} \quad \text{ for every pair of incomparable proper flats } F_1, F_2, \text{ and}
\end{equation}
\begin{equation}
\label{eq:ideal_description_3}
\tag{R3}
y_ex_F \quad \text{ for every $e \in E$ and for every proper flat $F$ not containing $e$.}
\end{equation}
Let $I = \{i_1,\dots,i_s\}$ be a (possibly empty) rank-$s$ independent set in $M_V$, and let 
\[\mathscr{F}:F_{1} \subsetneq F_{2} \subsetneq \cdots \subsetneq F_{t}\]
be a (possibly empty) flag of $t$ proper flats of $M_V$, such that $(I, \mathscr{F})$ is a compatible pair of $M_V$. We use the symbol $x_{I, \mathscr{F}}$ to denote the element
\[y_{i_1}y_{i_2}\cdots y_{i_s} x_{F_1}x_{F_2}\cdots x_{F_t} \in A^{s+t}(W_V).\]
The pullback map ${\kappa}^*: A^*(W_{\mathbb{C}^E}) \to A^*(W_V)$ sends $\overline{x}_F$ to $x_F$ and $\overline{y}_e$ to $y_e$.
\end{opchowring}

The degree map $\deg_{W_V}: A_0(W_V) \to \mathbb{Z}$ is described in \cite[Definition 2.15]{semismall}. If $(I, \mathscr{F})$ is a maximal compatible pair on $M_V$, then $\deg_{W_V}(x_{I, \mathscr{F}} \cap [W_V])=1.$

\begin{lemma}
\label{lem:chow_ring_computations}

Let $I = \{i_1,\dots,i_s\}$ be a (possibly empty) rank-$s$ independent set of $M_V$, and let 
\[\mathscr{F}:F_{s+1} \subsetneq F_{s+2} \subsetneq \cdots \subsetneq F_{s+t}\]
be a (possibly empty) flag of $t$ proper flats of $M_V$, such that $(I, \mathscr{F})$ is a nearly maximal compatible pair of $M_V$.
If $\mathscr{F}$ is non-empty, we have:
\begin{enumerate}
    \item[(I)] $\deg_{W_V}((x_{F_{s+1}} \cup x_{I, \mathscr{F}}) \cap [W_V]) = -1$, and
    \item[(II)] $\deg_{W_V}((x_{F_{s+k}} \cup x_{I, \mathscr{F}}) \cap [W_V]) = 0$ for all $2 \leq k \leq t$.
\end{enumerate}
Furthermore, let $a: \mathcal{L}_{\neq r} (M_V) \to \mathbb{Z}$ be a set function, and let $\alpha$ be the element
\[\sum_{F \in \mathcal{L}_{\neq r} (M_V)} a(F)x_F\]
in  $A^1(W_V)$. Then, we have:
\begin{enumerate}
    \item[(III)] $\deg_{W_V}((\alpha \cup x_{I, \mathscr{F}}) \cap [W_V]) = a(\cl(I)) - a(F_{s+1}),$
\end{enumerate}
setting $a(F_{s+1}) := 0$ if $\mathscr{F}$ is empty.
\end{lemma}

The proof of Lemma \ref{lem:chow_ring_computations} is lengthy, so it has been relegated to Appendix A.

Given a flat $F$ of $M_V,$ or the set difference of two flats $F_2 \setminus F_1$, let $\mathcal{L}_1(F)$ and $\mathcal{L}_1(F_2 \setminus F_1)$ denote the collection of rank-1 flats of $M_V$ that partition $F$ and $F_2 \setminus F_1$, respectively (cf. Lemma \ref{lem:flat_disjoint_union}).

\begin{lemma}
\label{lem:line_bundle_construction}
Let $a: \mathcal{L}_{\neq r}(M_V) \to \mathbb{Z}$ be a set function such that, for all proper flats $F$ of $M_V$, we have
\[a(F) \ =\sum_{G \in \mathcal{L}_1(E \setminus F)} \big( a(\emptyset) - a(G)\big).\]
Then, there exists a set function $b:E \to \mathbb{Z}$ such that
\[(\iota \circ \pi)^* \sum_{e \in E} b(e) \tilde{y}_e  \quad = \sum_{F \in \mathcal{L}_{\neq r}(M_V) }a(F)x_F.\]
\end{lemma}

\begin{proof}
    For each rank-1 flat $G$ of $M_V,$ choose a representative element $e_G \in G.$ Define $b: E \to \mathbb{Z}$ to be the function that sends the element $e_G$ to the integer $a(\emptyset) - a(G)$, and all other elements to 0. Then, by the commutativity of Diagram \ref{dia:commutative_square} and the identity $(\rho \circ \kappa)^* \tilde{y}_e = y_e$, we have
    \begin{equation*}
    \label{eq:line_bundle_construction}
    (\iota \circ \pi)^* \sum_{e \in E} b(e) \tilde{y}_e = (\rho \circ \kappa)^* \sum_{e \in E} b(e) \tilde{y}_e =
    \sum_{e \in E} b(e) y_e. 
    \end{equation*}
    By the set of relations (\ref{eq:ideal_description_1}) of $A^*(W_V)$, we have
    \[\sum_{e \in E} b(e) y_e = \sum_{e\in E} \left( b(e)\sum_{F \not\ni e} x_F \right)=\sum_{e\in E} \sum_{F \not\ni e} \left(b(e) x_F\right).\]
    We reindex the double sum to get
    \[\sum_{e\in E} \sum_{F \not\ni e} \left(b(e) x_F\right) = \sum_{F \in \mathcal{L}_{\neq r}(M_V)} \: \sum_{e \in E \setminus F} (b(e)x_F) = \sum_{F \in \mathcal{L}_{\neq r}(M_V)} \left( \sum_{e \in E \setminus F} b(e) \right) x_F.\]
    By how we defined $b$, we have
    \[\sum_{F \in \mathcal{L}_{\neq r}(M_V)} \left( \sum_{e \in E \setminus F} b(e) \right) x_F = \sum_{F \in \mathcal{L}_{\neq r}(M_V)} \left( \sum_{G \in \mathcal{L}_1(E \setminus F)} (a(\emptyset) - a(G)) \right) x_F.\]
    But by the hypothesis, we have
    \[\sum_{F \in \mathcal{L}_{\neq r}(M_V)} \left( \sum_{G \in \mathcal{L}_1(E \setminus F)} (a(\emptyset) - a(G)) \right) x_F = \sum_{F \in \mathcal{L}_{\neq r}(M_V)} a(F) x_F. \qedhere\]
\end{proof}

\begin{lemma}
\label{lem:alternate_equations}

Let $a: \mathcal{L}_{\neq r}(M_V) \to \mathbb{Z}$ be a set function such that, for all pairs of proper flats $F_1, F_2$ such that $F_1 \subsetneq F_2,$ and
$\rk(F_2) =\rk(F_1) + 1,$ we have
\[a(F_1) - a(F_2) \ = \sum_{G \in \mathcal{L}_1(F_2 \setminus F_1)}(a(\emptyset) - a(G)),\]
and for all corank-1 flats $F_3,$ we have
\[a(F_3) \ =\sum_{G \in \mathcal{L}_1(E \setminus F_3)}(a(\emptyset) - a(G)).\]
Then, for all proper flats $F$ of rank at least 1, we have
\[a(F) \ =\sum_{G \in \mathcal{L}_1(E \setminus F)} \big( a(\emptyset) - a(G)\big).\]
\end{lemma}
\begin{proof}
We do induction on the corank of $F$. The base case is when $F$ is a corank-$1$ flat; in this case, the conclusion of the lemma automatically follows from the hypothesis. For the inductive step, let the conclusion be true for corank-$k$ flats, and let $F$ be a corank-$(k+1)$ flat. Choose $F'$ to be a corank-$k$ flat that contains $F$. Then, by the hypothesis of the lemma, we have
\[a(F) - a(F') = \sum_{G \in \mathcal{L}_1(F' \setminus F)}(a(\emptyset) - a(G)), \]
and by the inductive hypothesis, we have
\[a(F')=\sum_{G \in \mathcal{L}_1(E \setminus F')} \big( a(\emptyset) - a(G)\big).\]
So we have
\begin{equation*}
\begin{split}
a(F)= \big(a(F) - a(F') \big) + a(F') & = \sum_{G \in \mathcal{L}_1(F' \setminus F)}(a(\emptyset) - a(G)) \quad + \sum_{G \in \mathcal{L}_1(E \setminus F')} \big( a(\emptyset) - a(G)\big) \\ & =  \sum_{G \in \mathcal{L}_1(E \setminus F)} \big( a(\emptyset) - a(G)\big),
\end{split}
\end{equation*}
completing the inductive step.
\end{proof}

\begin{definition}
\label{def:minkowski_weight}
Let $\Sigma$ be a unimodular fan, and let $\Sigma(k)$ denote the set of $k$-dimensional cones of $\Sigma$. A $k$-dimensional \textit{Minkowski weight} on $\Sigma$ is a set function $w: \Sigma(k) \to \mathbb{Z},$ such that, for every cone $\tau \in \Sigma(k-1),$ we have
\[\sum_{\sigma \supset \tau} w(\sigma) v_{\sigma, \tau} \in \Span(\tau),\]
where the sum is over all $\sigma \in \Delta(k)$ such that $\sigma \supset \tau$, and $v_{\sigma, \tau}$ denotes the first lattice point along the unique ray of $\sigma$ not contained in $\tau.$ Let $\MW_k(\Sigma)$ denote the group of $k$-dimensional Minkowski weights on $\Sigma.$
\end{definition}

\begin{theorem}
\label{thm:fulton_sturmfels_iso}
\cite[Theorem 3.1]{fulton_sturmfels}
Let $\Sigma$ be an $n$-dimensional, complete, unimodular fan, and let $X(\Sigma)$ denote its corresponding toric variety. For $0 \leq k \leq n$, we have an isomorphism from $A^{k}(X(\Sigma))$ to $\mathrm{MW}_{n-k}$ given by
\begin{equation*}
\label{eq:chow_minkowski_iso}
\alpha\mapsto \Big[\sigma \mapsto\deg_{X(\Sigma)}(\alpha \cap [V(\sigma)])\Big],
\end{equation*}
where $[V(\sigma)]$ denotes the Chow homology class of the torus orbit closure corresponding to the cone $\sigma.$
\end{theorem}

So, by Theorem \ref{thm:fulton_sturmfels_iso}, we can treat Minkowski weights on $\Delta_E$ and $\Gamma_E$ as operational Chow cohomology classes on $W_{\mathbb{C}^E}$ and $Y_{\mathbb{C}^E}$, respectively. Let $\overline{w}_{M_V}$ denote the $r$-dimensional Minkowski weight on the fan $\Delta_E$ defined by the rule
\[\overline{\sigma}_{I, \mathscr{F}} \mapsto 
\begin{cases}
    1 & (I, \mathscr{F}) \text{ is a compatible pair of } M_V \\
    0 & \text{otherwise}.
\end{cases}\]
(By \cite[Proposition 2.8]{semismall}, the weight $\overline{w}_{M_V}$ satisfies the ``balancing condition'' in Definition \ref{def:minkowski_weight}.)
For any subset $S \subseteq E,$ let $\tilde{w}_S$ denote the 1-dimensional Minkowski weight on the fan $\Gamma_E$ defined by the rule
\[\tilde{\sigma}_{S_1,S_2} \mapsto 
\begin{cases}
    1 & S_1 \cup S_2 = \{e\} \text{ for some } e \in S \\   
    0 & \text{otherwise}.
\end{cases}\]

\begin{lemma}
\label{lem:pushforward_wonderful_model}
\cite[Corollary 5.11]{stellahedral}
We have $\kappa_*[W_V] = \overline{w}_{M_V} \cap [W_{\mathbb{C}^E}]$.
\end{lemma}

\begin{lemma}
\label{lem:pushforward_weight}
Let $w$ be a 1-dimensional Minkowski weight on the fan $\Delta_E.$ Then,
\[\rho_*(w \cap [W_{\mathbb{C}^E}]) = w' \cap [Y_{\mathbb{C}^E}],\]
where $w'$ is the 1-dimensional Minkowski weight on the fan $\Delta_E$ defined by the rules
\[w'(\tilde{\sigma}_{e, \emptyset}) = w(\overline{\sigma}_{e,\emptyset}) \quad \text{and} \quad w'(\tilde{\sigma}_{\emptyset, e}) \mapsto w(\overline{\sigma}_{e,\emptyset}).\]
\end{lemma}
\begin{proof}
The pushforward $\rho_*(w \, \cap \, [W_{\mathbb{C}^E}])$ is a degree-1 Chow homology class in $A_1(Y_{\mathbb{C}^E}).$ The variety is $Y_{\mathbb{C}^E}$ is smooth, so by  the Poincar\'e duality isomorphism, there exists a operational Chow cohomology class in $A^{n-1}(Y_{\mathbb{C}^E})$ that is Poincar\'e dual to $\rho_*(w \cap [W_{\mathbb{C}^E}])$. Using Theorem \ref{thm:fulton_sturmfels_iso}, we think of this cohomology class as a 1-dimensional Minkowski weight $w'.$ By Theorem \ref{thm:fulton_sturmfels_iso}, we have
\[w'(\tilde{\sigma}_{e, \emptyset}) = \deg_{Y_{\mathbb{C}^E}}(w' \cap [V(\tilde{\sigma}_{e, \emptyset})]).\]
The cohomology class $\tilde{y}_e \in A^1(W_V)$ is the Poincar\'e dual of the homology class $[V(\tilde{\sigma}_{e, \emptyset})],$ so we have
\[\deg_{Y_{\mathbb{C}^E}}(w' \cap [V(\tilde{\sigma}_{e, \emptyset})]) = \deg_{Y_{\mathbb{C}^E}}(w' \cap ( \tilde{y}_e \cap [Y_{\mathbb{C}^E}])).\]
By the $A^*(Y_{\mathbb{C}^E})$-module structure of $A_*(Y_{\mathbb{C}^E})$ and the commutativity of $A^*(Y_{\mathbb{C}^E})$, we have
\[\deg_{Y_{\mathbb{C}^E}}(w' \cap ( \tilde{y}_e \cap [Y_{\mathbb{C}^E}])) =
\deg_{Y_{\mathbb{C}^E}}(\tilde{y}_e \cap ( w' \cap [Y_{\mathbb{C}^E}])).\]
Then, by the construction of $w'$, we have
\[\deg_{Y_{\mathbb{C}^E}}(\tilde{y}_e \cap ( w' \cap [Y_{\mathbb{C}^E}])) = \deg_{Y_{\mathbb{C}^E}}(\tilde{y}_e \cap  \rho_*(w \cap [W_{\mathbb{C}^E}])).\]
By the projection formula, we have
\[\deg_{Y_{\mathbb{C}^E}}(\tilde{y}_e \cap  \rho_*(w \cap [W_{\mathbb{C}^E}])) = \deg_{Y_{\mathbb{C}^E}}\big(\rho_*(\rho^*\tilde{y}_e \cap  (w \cap [W_{\mathbb{C}^E}]))\big),\]
and we use the identities $\deg_{Y_{\mathbb{C}^E}} \circ \; \rho_* = \deg_{W_{\mathbb{C}^E}}$ and $\rho^*\tilde{y}_e = \overline{y}_e$ to get
\[\deg_{Y_{\mathbb{C}^E}}(\rho_*(\rho^*\tilde{y}_e \cap  (w \cap [W_{\mathbb{C}^E}]))) = \deg_{W_{\mathbb{C}^E}}(\overline{y}_e \cap  (w \cap [W_{\mathbb{C}^E}])).\]
By the $A^*(W_{\mathbb{C}^E})$-module structure of $A_*(W_{\mathbb{C}^E})$ and the commutativity of $A^*(W_{\mathbb{C}^E})$, we have
\[\deg_{W_{\mathbb{C}^E}}(\overline{y}_e \cap  (w \cap [W_{\mathbb{C}^E}])) = \deg_{W_{\mathbb{C}^E}}(w \cap  (\overline{y}_e \cap [W_{\mathbb{C}^E}])), \]
and since $\overline{y}_e \cap [W_{\mathbb{C}^E}] = [V(\overline{\sigma}_{e, \emptyset})],$ we get
\[\deg_{W_{\mathbb{C}^E}}(w \cap  (\overline{y}_e \cap [W_{\mathbb{C}^E}])) = \deg_{W_{\mathbb{C}^E}}(w \cap [V(\overline{\sigma}_{e, \emptyset})]).\]
But by Theorem \ref{thm:fulton_sturmfels_iso}, we have that
\[\deg_{W_{\mathbb{C}^E}}(w \cap [V(\overline{\sigma}_{e, \emptyset})]) = w(\overline{\sigma}_{e, \emptyset}),\]
so we have found that
\[w'(\tilde{\sigma}_{e, \emptyset}) = w(\overline{\sigma}_{e,\emptyset}).\]
The Chow homology class $[V(\tilde{\sigma}_{\emptyset, e})]$ is the same as $[V(\tilde{\sigma}_{e, \emptyset})]$,
so the rule $w'(\tilde{\sigma}_{\emptyset, e}) = w(\overline{\sigma}_{e,\emptyset})$ follows by a near-identical argument.    
\end{proof}

\begin{lemma}
\label{lem:pushforward_formula}
Let $(I, \mathscr{F})$ be a nearly maximal compatible pair of $M_V$. Let $F$ denote the smallest flat in $\mathscr{F}$. (If $\mathscr{F}$ is empty, then set $F := E.$) Then, we have
\[(\iota \circ \pi)_*(x_{I,\mathscr{F}} \cap [W_V]) = \tilde{w}_{F \setminus \cl(I)} \cap [Y_{\mathbb{C}^E}].\]
\end{lemma}

\begin{proof}
We show that the left-hand side equals the right-hand side. By the commutativity of Diagram \ref{dia:commutative_square}, we have
\[(\iota \circ \pi)_*(x_{I,\mathscr{F}} \cap [W_V]) = (\rho \circ \kappa)_*(x_{I,\mathscr{F}} \cap [W_V]).\]
Since $\kappa^*(\overline{x}_{I,\mathscr{F}}) = x_{I,\mathscr{F}},$ we have
\[(\rho \circ \kappa)_*(x_{I,\mathscr{F}} \cap [W_V]) = (\rho \circ \kappa)_*(\kappa^*(\overline{x}_{I,\mathscr{F}}) \cap [W_V]),\]
and by the projection formula, we have
\[(\rho \circ \kappa)_*(\kappa^*(\overline{x}_{I,\mathscr{F}}) \cap [W_V]) = \rho_*(\overline{x}_{I, \mathscr{F}} \cap \kappa_*[W_V]).\]
By Lemma \ref{lem:pushforward_wonderful_model}, we have
\[\rho_*(\overline{x}_{I, \mathscr{F}} \cap \kappa_*[W_V]) = \rho_*(\overline{x}_{I, \mathscr{F}} \cap (\overline{w}_{M_V} \cap [W_{\mathbb{C}^E}])),\]
and by the $A^*(W_{\mathbb{C}^E})$-module structure of $A_*(W_{\mathbb{C}^E}),$ we have
\[\rho_*(\overline{x}_{I, \mathscr{F}} \cap (\overline{w}_{M_V} \cap [W_{\mathbb{C}^E}])) = \rho_*((\overline{x}_{I, \mathscr{F}} \cup \overline{w}_{M_V}) \cap [W_{\mathbb{C}^E}]).\]
By Lemma \ref{lem:pushforward_weight},
\[\rho_*((\overline{x}_{I, \mathscr{F}} \cup \overline{w}_{M_V}) \cap [W_{\mathbb{C}^E}]) = w' \cap [Y_{\mathbb{C}^E}],\]
where $w'$ is the 1-dimensional Minkowski weight on $\Gamma_E$ defined by the rules
\[\tilde{\sigma}_{e, \emptyset} \mapsto (\overline{x}_{I, \mathscr{F}} \cup \overline{w}_{M_V})(\overline{\sigma}_{e,\emptyset}) \quad \text{and} \quad \tilde{\sigma}_{\emptyset, e} \mapsto (\overline{x}_{I, \mathscr{F}} \cup \overline{w}_{M_V})(\overline{\sigma}_{e,\emptyset}),\]
if we think about the cohomology class $(\overline{x}_{I, \mathscr{F}} \cup \overline{w}_{M_V}) \in A^{n-1}(W_{\mathbb{C}^E})$ as a 1-dimensional Minkowski weight on $\Delta_E.$
So, we compute $(\overline{x}_{I, \mathscr{F}} \cup \overline{w}_{M_V})(\overline{\sigma}_{e,\emptyset}).$ By Theorem \ref{thm:fulton_sturmfels_iso}, we have
\[(\overline{x}_{I, \mathscr{F}} \cup \overline{w}_{M_V})(\overline{\sigma}_{e,\emptyset}) = \deg_{W_{\mathbb{C}^E}}((\overline{x}_{I, \mathscr{F}} \cup \overline{w}_{M_V}) \cap [V(\overline{\sigma}_{e,\emptyset})]).\]
The Poincar\'e dual of the homology class $[V(\overline{\sigma}_{e,\emptyset})]$ is the cohomology class $\overline{y}_e$, so we have,
\[\deg_{W_{\mathbb{C}^E}}((\overline{x}_{I, \mathscr{F}} \cup \overline{w}_{M_V}) \cap [V(\overline{\sigma}_{e,\emptyset})]) = \deg_{W_{\mathbb{C}^E}}((\overline{x}_{I, \mathscr{F}} \cup \overline{w}_{M_V}) \cap (\overline{y}_e \cap [W_{\mathbb{C}^E}])),\]
and by the $A^*(W_{\mathbb{C}^E})$-module structure of $A_*(W_{\mathbb{C}^E})$ and the commutativity of $A^*(W_{\mathbb{C}^E})$, we have
\begin{equation}
\label{eq:pushforward_formula_1}
\begin{split}
& \deg_{W_{\mathbb{C}^E}}((\overline{x}_{I, \mathscr{F}} \cup \overline{w}_{M_V}) \cap (\overline{y}_e \cap [W_{\mathbb{C}^E}])) =  \\ & \deg_{W_{\mathbb{C}^E}}((\overline{w}_{M_V} \cup \overline{x}_{I, \mathscr{F}} \cup \overline{y}_e)\cap [W_{\mathbb{C}^E}])).
\end{split}
\end{equation}
We do cases on the element $e$ to compute the value of the right-hand side of Equation (\ref{eq:pushforward_formula_1}).

\noindent
\textit{Case 1: $e \in I$.} By \cite[Lemma 2.11]{semismall}, we have
\[\overline{x}_{I, \mathscr{F}} \cup \overline{y}_e =  0,\]
so
\[\deg_{W_{\mathbb{C}^E}}((\overline{w}_{M_V} \cup \overline{x}_{I, \mathscr{F}} \cup \overline{y}_e)\cap [W_{\mathbb{C}^E}])) = 0.\]

\noindent
\textit{Case 2: $e \in \cl(I) \setminus I$.} Let $I' = I \cup e.$ The set $I'$ is a dependent set of $M_V,$ so we have that $(I', \mathscr{F})$ is a compatible pair of the Boolean matroid, but not a compatible pair of $M_V.$
Thus,
\begin{equation*}
\deg_{W_{\mathbb{C}^E}}((\overline{w}_{M_V} \cup \overline{x}_{I, \mathscr{F}} \cup \overline{y}_e)\cap [W_{\mathbb{C}^E}])) = \deg_{W_{\mathbb{C}^E}}((\overline{w}_{M_V}
\cup \overline{x}_{I', \mathscr{F}})\cap [W_{\mathbb{C}^E}])),
\end{equation*}
and by the $A^*(W_{\mathbb{C}^E})$-module structure of $A_*(W_{\mathbb{C}^E})$ and Poincar\'e duality, we have
\begin{equation*}
\begin{split}
\deg_{W_{\mathbb{C}^E}}((\overline{w}_{M_V}
\cup \overline{x}_{I', \mathscr{F}})\cap [W_{\mathbb{C}^E}])) & = \deg_{W_{\mathbb{C}^E}}(\overline{w}_{M_V} \cap (\overline{x}_{I', \mathscr{F}}\cap [W_{\mathbb{C}^E}])) \\ & = \deg_{W_{\mathbb{C}^E}}(\overline{w}_{M_V} \cap [V(\sigma_{I', \mathscr{F}})]).
\end{split}
\end{equation*}
By Theorem \ref{thm:fulton_sturmfels_iso}, we have
\[\deg_{W_{\mathbb{C}^E}}(\overline{w}_{M_V} \cap [V(\sigma_{I', \mathscr{F}})]) = \overline{w}_{M_V}(\sigma_{I', \mathscr{F}}) = 0.\]

\noindent
\textit{Case 3: $e \in F \setminus \cl(I)$.}
Let $I' = I \cup e.$ Since $I'$ is an independent set of $M_V,$ we have that $(I', \mathscr{F})$ is both a compatible pair of the Boolean matroid and a compatible pair of $M_V.$ By an identical argument to the one for Case 2, we have
\[\deg_{W_{\mathbb{C}^E}}((\overline{w}_{M_V} \cup \overline{x}_{I, \mathscr{F}} \cup \overline{y}_e)\cap [W_{\mathbb{C}^E}])) = \overline{w}_{M_V}(\sigma_{I', \mathscr{F}}).\]
But this time,
\[\overline{w}_{M_V}(\sigma_{I', \mathscr{F}}) = 1.\]

\noindent
\textit{Case 4: $e \in E \setminus F$.} By the relations (\textoverline{R3}), we have
\[\overline{x}_{I, \mathscr{F}} \cup \overline{y}_e =  0,\]
so
\[\deg_{W_{\mathbb{C}^E}}((\overline{w}_{M_V} \cup \overline{x}_{I, \mathscr{F}} \cup \overline{y}_e)\cap [W_{\mathbb{C}^E}])) = 0.\]

To recap,
\begin{equation*}
(\overline{x}_{I, \mathscr{F}} \cup \overline{w}_{M_V})(\overline{\sigma}_{e,\emptyset}) =
\begin{cases}
    1 & e \in F \setminus \cl(I) \\
    0 & \text{otherwise}.
\end{cases}
\end{equation*}
Hence $w' = \tilde{w}_{F \setminus \cl(I)},$ and we are done.
\end{proof}

\begin{lemma}
\label{lem:pushforward_equality}
Let $(I, \mathscr{F})$ be a nearly maximal compatible pair of $M_V$. Let $F$ denote the smallest flat in $\mathscr{F}$. (If $\mathscr{F}$ is empty, then set $F := E.$) For each $G \in \mathcal{L}_1(F \setminus \cl(I))$, let
\[\mathscr{F}_G:F_{G,1} \subsetneq F_{G,2} \subsetneq \cdots \subsetneq F_{G,r-1}\]
be a flag of proper flats of $M_V$ such that $F_{G,1}$ is the flat $G$, $F_{G,2}$ is a flat of rank $2,\ldots,$ and $F_{G,r-1}$ is a flat of rank $r-1$. Then, we have the following equality in the Chow group $A_1(Y_V)$:

\[\pi_*(x_{I,\mathscr{F}} \cap [W_V])=\sum_{G \in \mathcal{L}_1(F \setminus \cl(I))}\pi_*\left(x_{\emptyset, \mathscr{F}_G} \cap [W_V] \right).\]
\end{lemma}

\begin{proof}
By the injectivity of $\iota_*$ by Corollary \ref{cor:pushforward_injective}, it suffices to show that we have the equality
\begin{equation}
\label{eq:pushforward_equality_1}
{(\iota \circ \pi)}_*(x_{I, \mathscr{F}} \cap [W_V])=\sum_{G \in \mathcal{L}_1(F \setminus \cl(I))}{(\iota \circ \pi)}_*\left( x_{\emptyset, \mathscr{F}_G} \cap [W_V]\right).
\end{equation}
By Lemma \ref{lem:pushforward_formula}, the left-hand side of Equation (\ref{eq:pushforward_equality_1}) is
\begin{equation*}
\label{eq:pushforward_equality_2}
\tilde{w}_{F \setminus \cl(I)} \cap [Y_{\mathbb{C}^E}],
\end{equation*}
and the right-hand side of Equation (\ref{eq:pushforward_equality_1}) is
\begin{equation*}
\label{eq:pushforward_equality_3}
\sum_{G \in \mathcal{L}_1(F \setminus \cl(I))}\big(\tilde{w}_G \cap [Y_{\mathbb{C}^E}]\big).
\end{equation*}
It is straightforward to check that
\[\tilde{w}_{F \setminus \cl(I)} = \sum_{G \in \mathcal{L}_1(F \setminus \cl(I))}\tilde{w}_G,\]
so we are done.
\end{proof}

\begin{lemma}
\label{lem:equality_of_degrees}
    Let $L$ be a line bundle on $Y_V$ and consider the first Chern class $c_1(\pi^*L) \in A^1(W_V)$. Let $\alpha , \beta \in A^{r-1}(W_V)$ such that $\pi_*(\alpha \cap [W_V]) = \pi_*(\beta \cap [W_V]).$  Then,
    \[\deg_{W_V}((c_1(\pi^*L) \cup \alpha) \cap[W_V]) = \deg_{W_V}((c_1(\pi^*L) \cup \beta) \cap [W_V]).\]
\end{lemma}

\begin{proof}   
    By the hypothesis, we have
    \[c_1(L) \cap \pi_*(\alpha \cap [W_V]) = c_1(L) \cap \pi_*(\beta \cap [W_V]).\]
    The projection formula gives us
    \[\pi_*(\pi^*(c_1(L)) \cap (\alpha \cap [W_V])) = \pi_*(\pi^*(c_1(L)) \cap (\beta\cap [W_V])),\]
    and since the pullback commutes with taking the Chern class, we get
    \[\pi_*(c_1(\pi^*L) \cap (\alpha \cap [W_V])) = \pi_*(c_1(\pi^*L) \cap (\beta\cap [W_V])).\]
    By the $A^*(W_{V})$-module structure of $A_*(W_{V}),$ we have
    \[\pi_*((c_1(\pi^*L) \cup  \alpha) \cap [W_V]) = \pi_*((c_1(\pi^*L) \cup \beta) \cap [W_V]).\]
    By taking the $\deg_{Y_V}$ of both sides, and noting that $\deg_{W_V} = \deg_{Y_V} \circ \; \pi_*,$ we see that
    \[\deg_{W_V}((c_1(\pi^*L) \cup  \alpha) \cap [W_V]) = \deg_{W_V}((c_1(\pi^*L) \cup \beta) \cap [W_V]). \qedhere\] 
\end{proof}

\begin{lemma}
\label{lem:conincidence_of_line_bundles}
    Let $L_1$ be a line bundle on $Y_V$. There exists a line bundle $L_2$ on $Y_{\mathbb{C}^E}$ such that $(\iota \circ \pi)^*L_2 \cong {\pi}^*L_1$.
\end{lemma}

\begin{proof}
If $X$ and $Y$ are smooth varieties and $f: X \to Y$ is a scheme morphism, then taking the first Chern class of line bundles gives us isomophisms $\Pic(X) \cong A^1(X)$ and $\Pic(Y) \cong A^1(Y)$, and these isomorphisms are compatible with pulling back. Both $W_V$ and $Y_{\mathbb{C}^E}$ are smooth varieties, and $\iota \circ \pi$ is a morphism from $W_V$ to $Y_{\mathbb{C}^E}$, so it suffices to show that there exists an element $\beta \in A^1(Y_{\mathbb{C}^E})$ such that $(\iota \circ\pi)^* \beta = c_1(\pi^*L_1).$ 

Using the set of the relations (R1) in $A^*(W_V),$ we write $c_1(\pi^*L_1)\in A^1(W_V)$ as a $\mathbb{Z}$-linear combination of the symbols $x_F$, and let $a: \mathcal{L}_{\neq r}(M_V) \to \mathbb{Z}$ be the function that takes the proper flat $F$ to the coefficient of $x_F$ in $c_1(\pi^*L_1).$ Then, by combining Lemma \ref{lem:line_bundle_construction} and Lemma \ref{lem:alternate_equations}, it suffices to show that, for all pairs of proper flats $F_1, F_2$ such that $F_1 \subsetneq F_2,$ and
$\rk(F_2) =\rk(F_1) + 1,$ we have
\begin{equation}
\label{eq:coincidence_of_line_bundles_1}
a(F_1) - a(F_2) = \sum_{G \in \mathcal{L}_1(F_2 \setminus F_1)}(a(\emptyset) - a(G)),
\end{equation}
and for all corank-1 flats $F_3,$ we have
\begin{equation}
\label{eq:coincidence_of_line_bundles_2}
a(F_3)=\sum_{G \in \mathcal{L}_1(E \setminus F_3)}(a(\emptyset) - a(G)).
\end{equation}

Let such a pair $F_1,F_2$ be given.
Let $s = \rk(F_1)$ and $t = (r-1) - s$. Choose $I$ to be a basis of $M | F_1,$ and choose an flag
\[\mathscr{F}:F_{s+1} \subsetneq F_{s+2} \subsetneq \cdots \subsetneq F_{s+t}\]
of proper flats of $M_V$ such that $F_{s+1}$ is the flat $F_2,$ $F_{s+2}$ is a flat of rank $s+2,\ldots,$ and $F_{s+t}$ is a flat of rank $s+t$. We note that $(I, \mathscr{F})$ is a nearly maximal pair of $M_V.$ For each $G \in \mathcal{L}_1(F_2 \setminus \cl(I))$, choose a flag
\[\mathscr{F}_G:F_{G,1} \subsetneq F_{G,2} \subsetneq \cdots \subsetneq F_{G,r-1}\]
of proper flats of $M_V$ such that $F_{G,1}$ is the flat $G$, $F_{G,2}$ is a flat of rank $2,\ldots,$ and $F_{G,r-1}$ is a flat of rank $r-1$. 
By Lemma \ref{lem:pushforward_equality}, we have
\[\pi_*( x_{I, \mathscr{F}} \cap [W_V]) = \sum_{G \in \mathcal{L}_1(F_2 \setminus F_1)}\pi_*( x_{\emptyset, \mathscr{F}_G} \cap[W_V]),\]
and by Lemma \ref{lem:equality_of_degrees}, we have
\begin{equation}
\label{eq:coincidence_of_line_bundles_3}
\begin{split}
& \deg_{W_V}((c_1(\pi^*L) \cup x_{I, \mathscr{F}}) \cap [W_V]) = \\ \sum_{G \in \mathcal{L}_1(F_2 \setminus F_1)} &\deg_{W_V}((c_1(\pi^*L) \cup x_{\emptyset, \mathscr{F}_G}) \cap[W_V]).
\end{split}
\end{equation}
By applying Lemma \ref{lem:chow_ring_computations}, Equation (\ref{eq:coincidence_of_line_bundles_3}) becomes Equation (\ref{eq:coincidence_of_line_bundles_1}). 

Let a corank-1 flat $F_3$ be given. Choose $I$ to be a basis of $M|F_3,$ and let $\mathscr{F}$ be the empty flag. For each $G \in \mathcal{L}_1(E \setminus F_3)$, choose a flag
\[\mathscr{F}_G:F_{G,1} \subsetneq F_{G,2} \subsetneq \cdots \subsetneq F_{G,r-1}\]
of proper flats of $M_V$ such that $F_{G,1}$ is the flat $G$, $F_{G,2}$ is a flat of rank $2,\ldots,$ and $F_{G,r-1}$ is a flat of rank $r-1$.
By Lemma \ref{lem:pushforward_equality}, we have
\[\pi_*( x_{I, \mathscr{F}} \cap [W_V]) = \sum_{G \in \mathcal{L}_1(E \setminus F_3)}\pi_*( x_{\emptyset, \mathscr{F}_G} \cap[W_V]),\]
and by Lemma \ref{lem:equality_of_degrees}, we have
\begin{equation}
\label{eq:coincidence_of_line_bundles_4}
\begin{split}
& \deg_{W_V}((c_1(\pi^*L) \cup x_{I, \mathscr{F}}) \cap [W_V]) \\ = \sum_{G \in \mathcal{L}_1(E \setminus F_3)} & \deg_{W_V}((c_1(\pi^*L) \cup x_{\emptyset, \mathscr{F}_G}) \cap[W_V]).
\end{split}
\end{equation}
By applying Lemma \ref{lem:chow_ring_computations}, Equation (\ref{eq:coincidence_of_line_bundles_4}) becomes Equation (\ref{eq:coincidence_of_line_bundles_2}).
\end{proof}

\begin{lemma}
\label{lem:pullback_injective}

The pullback map $\pi^*: \Pic(Y_V) \to \Pic(W_V)$ is injective.
\end{lemma}

\begin{proof}

    Let $L$ be a line bundle on $Y_V$ such that ${\pi}^*L \cong \mathcal{O}_{W_V}.$ It will suffice to show that $L \cong \mathcal{O}_{Y_V}$. We start by writing the isomorphism
    \begin{equation*}
    \label{eq:injective1}
    \mathcal{O}_{W_V} \otimes \pi^*L \cong \mathcal{O}_{W_V}.
    \end{equation*}
    We take the direct image sheaf of both sides
    \begin{equation*}
    \label{eq:injective2}
    \pi_*(\mathcal{O}_{W_V} \otimes \pi^*L) \cong \pi_*(\mathcal{O}_{W_V})
    \end{equation*}
    and use the projection formula on the left-hand side to get
    \begin{equation}
    \label{eq:injective3}
    \pi_*(\mathcal{O}_{W_V}) \otimes L \cong \pi_*(\mathcal{O}_{W_V}).
    \end{equation}
    By \cite[Theorem 4.1]{crowley_proudfoot}, $Y_V$ is normal. Hence, by Zariski's Main Theorem, we have that $\pi_*\mathcal{O}_{W_V} = \mathcal{O}_{Y_V}.$ Thus, Isomorphism (\ref{eq:injective3}) becomes
    \begin{equation*}
    \label{eq:injective4}
    \mathcal{O}_{Y_V} \otimes L \cong \mathcal{O}_{Y_V}. \qedhere
    \end{equation*}      
    
\end{proof}

\begin{proof}[Proof of Proposition \ref{prop:line_bundles}]
Let $L_1$ be a line bundle on $Y_V$. Lemma \ref{lem:conincidence_of_line_bundles} tells us that there exists a line bundle $L_2$ on $Y_{\mathbb{C}^E}$ such that $(\iota \circ \pi)^*L_2 \cong \pi^*L_1$. Since $\pi^*$ is injective by Lemma \ref{lem:pullback_injective}, we get that $\iota^*L_2 \cong L_1,$ and we are done.
\end{proof}

\section{Answering the titular question}

The objective of this section is to prove Theorem \ref{thm:gorenstein_condition}. 

\begin{lemma}
\label{lem:multiplication_identity}
For all $e\in E$, we have the following equality in the Chow group $A_*(Y_{\mathbb{C}^E})$:
\begin{equation*}
\label{eq:multiplication_identity_1}
\tilde{y}_e \cap \iota_*[Y_V] = \sum_{F \in \mathcal{H}_e(M_V)} \iota_*[\overline{Z}_{V, F}].
\end{equation*}
\end{lemma}

\begin{proof}
    We will show that the left-hand side equals the right-hand side. By Lemma \ref{lem:pushforward_descr}, we have
    \begin{equation*}
    \label{eq:multiplication_identity_2}    
    \tilde{y}_e \cap \iota_*[Y_V]= \tilde{y}_e \cap \left( \sum_{B \in \mathcal{B}(M_V)} [\overline{Z}_{\mathbb{C}^E,B}] \right) = \sum_{B \in \mathcal{B}(M_V)} \big( \tilde{y}_e \cap [\overline{Z}_{\mathbb{C}^E,B}] \big).
    \end{equation*}
    The Poincar\'e dual of $[\overline{Z}_{\mathbb{C}^E,B}] \in A_{r}(Y_{\mathbb{C}^E})$ is $\tilde{y}_{E \setminus B} \in A^{n-r}(Y_{\mathbb{C}^E}),$ so we have
    \[\sum_{B \in \mathcal{B}(M_V)} \big( \tilde{y}_e \cap [\overline{Z}_{\mathbb{C}^E,B}] \big)= \sum_{B \in \mathcal{B}(M_V)} \big( \tilde{y}_e \cap (\tilde{y}_{E \setminus B} \cap [Y_{\mathbb{C}^E}])\big),\]
    and by the $A^*(Y_{\mathbb{C}^E})$-module structure of $A_*(Y_{\mathbb{C}^E}),$ we have
    \[\sum_{B \in \mathcal{B}(M_V)} \big( \tilde{y}_e \cap (\tilde{y}_{E \setminus B} \cap [Y_{\mathbb{C}^E}])\big) = \sum_{B \in \mathcal{B}(M_V)} \big( (\tilde{y}_e \cup \tilde{y}_{E \setminus B}) \cap[Y_{\mathbb{C}^E}]\big).\]
    Looking at our presentation of $A^*(Y_{\mathbb{C}^E}),$ we see that
    \begin{equation*}
    \tilde{y}_e \cup \tilde{y}_{E \setminus B}=
    \begin{cases}
    \tilde{y}_{E \setminus (B \setminus e)} & e \in B \\
    0 & e \notin B.
    \end{cases}        
    \end{equation*}
    Thus, we have
    \[\sum_{B \in \mathcal{B}(M_V)} \big( (\tilde{y}_e \cup \tilde{y}_{E \setminus B}) \cap[Y_{\mathbb{C}^E}]\big) = \sum_{B' \in \mathcal{B}_e(M_V)} \big( \tilde{y}_{E \setminus B'} \cap[Y_{\mathbb{C}^E}]\big),\]
    where $\mathcal{B}_e(M_V)$ denotes the set
    \[\{B \setminus e \, : \, B \in \mathcal{B}(M_V), \: e\in B\}.\]
    The cohomology class $\tilde{y}_{E \setminus B'}$ is Poincar\'e dual to the homology class $[\overline{Z}_{\mathbb{C}^E,B'}],$ so we have
    \[\sum_{B' \in \mathcal{B}_e(M_V)} \big( \tilde{y}_{E \setminus B'} \cap[Y_{\mathbb{C}^E}]\big) = \sum_{B' \in \mathcal{B}_e(M_V)} [\overline{Z}_{\mathbb{C}^E,B'}].\]    
    Then, by Lemma \ref{lem:matroid_set_equality}, we have
    \begin{equation*}
    \label{eq:multiplication_identity_5}
    \sum_{B' \in \mathcal{B}_e(M_V)} [\overline{Z}_{\mathbb{C}^E,B'}]=\sum_{F \in \mathcal{H}_e(M_V)}\sum_{B \in \mathcal{B}(M_V | F)} [\overline{Z}_{\mathbb{C}^E, B}].
    \end{equation*}
    Finally, by another application of Lemma \ref{lem:pushforward_descr}, we have    
    \[ \sum_{F \in \mathcal{H}_e(M_V)}\sum_{B \in \mathcal{B}(M_V | F)} [\overline{Z}_{\mathbb{C}^E, B}]=\sum_{F \in \mathcal{H}_e(M_V)} \iota_*[\overline{Z}_{V, F}]. \qedhere\]
\end{proof}

Now, we compute the Cartier divisor class group of $Y_V$. For a normal variety $X$, there is an injective map from $\CaCl(X)$ to $\Cl(X).$ In this paper, we identify $\CaCl(X)$ with its image in $\Cl(X).$

\begin{proposition}
\label{prop:cartier_class_group}
    The Cartier divisor class group $\CaCl(Y_V)$, viewed as a subgroup of $\Cl(Y_V)$, is generated by the set
    \[ \left\{ \sum_{F \in \mathcal{H}_e(M_V)} [\overline{Z}_{V,F} ] \right\}_{e\in E}.\]
\end{proposition}

\begin{proof}
    Since $Y_{\mathbb{C}^E}$ is smooth, $\CaCl(Y_{\mathbb{C}^E}) = \Cl(Y_{\mathbb{C}^E}).$  Hence, by Lemma \ref{lem:chow_group}, we have that $\CaCl(Y_{\mathbb{C}^E})$ is freely generated by the set
    \[ \left\{ [\overline{Z}_{\mathbb{C}^E,E\setminus e} ] \right\}_{e\in E}.\]

    There is a surjective group homomorphism from $\CaCl(Y_{\mathbb{C}^E})$ to $\iota_*(\CaCl(Y_V))$ given by the following procedure: let $[D]$ be the class of a Cartier divisor $D$ on $Y_{\mathbb{C}^E}.$ Because $Y_{\mathbb{C}^E}$ is an integral scheme, the map from $\CaCl(Y_{\mathbb{C}^E})$ to $\Pic(Y_{\mathbb{C}^E})$ given by sending $[D]$ to the line bundle $L := \mathcal{O}_{Y_{\mathbb{C}^E}}(D)$ is an isomorphism. By Proposition \ref{prop:line_bundles}, we have a surjective group homomorphism from $\Pic(Y_{\mathbb{C}^E})$ to $\Pic(Y_V)$ by sending $L$ to $L|_{Y_V}.$ Then, we have an isomorphism from $\Pic(Y_V)$ to $\CaCl(Y_V)$; since $Y_V$ is an integral scheme, $L|_{Y_V}$ is isomorphic to $\mathcal{O}_{Y_V}(C)$ for some Cartier divisor $C$ on $Y_V,$ so we send $L|_{Y_V}$ to $[C].$ Finally, we send $[C]$ to $\iota_*[C].$ We apply this procedure to the generators of $\CaCl(Y_{\mathbb{C}^E})$ to get generators for $\iota_*(\CaCl(Y_V)).$
    
    Let $[\overline{Z}_{\mathbb{C}^E,E\setminus e}]$ be a generator of $\CaCl(Y_{\mathbb{C}^E})$. Say that $\alpha$ is the element in $\iota_*(\CaCl(Y_V)) \subseteq A_{r-1}(Y_{\mathbb{C}^E})$ that we get when we apply the procedure in the paragraph above to this generator. Then, \cite[Chapter 2]{intersection_theory} gives us a Chow cohomological way to compute $\alpha $:
    \begin{equation}
    \label{eq:cartier_class_group_1}
    c_1(L) \cap \iota_*[Y_V] = \alpha,
    \end{equation}
    where $L :=\mathcal{O}_{Y_{\mathbb{C}^E}}(\overline{Z}_{\mathbb{C}^E,E\setminus e}).$
    But since
    $c_1(L) \cap [Y_{\mathbb{C}^E}] = [\overline{Z}_{\mathbb{C}^E,E\setminus e}],$ we have that $c_1(L) \in A^1(Y_{\mathbb{C}^E})$ is the Poincar\'e dual of $[\overline{Z}_{\mathbb{C}^E,E\setminus e}],$ i.e., $\tilde{y}_e.$ So Equation (\ref{eq:cartier_class_group_1}) becomes
    \[\tilde{y}_e \cap \iota_*[Y_V] = \alpha.\]
    In conclusion, $\iota_*(\CaCl(Y_V))$ is generated by
    \[ \left\{ \tilde{y}_e \cap \iota_*[Y_V]\right\}_{e\in E}.\]    
    By Lemma \ref{lem:multiplication_identity}, we have
    \[\tilde{y}_e \cap \iota_*[Y_V] = \sum_{F \in \mathcal{H}_e(M_V)} \iota_*[\overline{Z}_{V,F} ], \] 
    so we have that $\iota_*(\CaCl(Y_V))$ is generated by the set
    \[ \left\{ \sum_{F \in \mathcal{H}_e(M_V)} \iota_*[\overline{Z}_{V,F} ] \right\}_{e\in E}.\]
    Since $\iota_*$ is injective by Corollary \ref{cor:pushforward_injective},
    we are done.    
\end{proof}

\begin{proposition}
\label{prop:canonical_divisor}
\cite[Proposition 4.4]{crowley_proudfoot}
The Weil divisor
\[\sum_{F \in \mathcal{H}(M_V)}-2\,\overline{Z}_{V,F}\]
is a canonical divisor of $Y_V.$
\end{proposition}

\begin{proof}[Proof of Theorem \ref{thm:gorenstein_condition}]

The algebraic variety $Y_V$ is Gorenstein if and only if the Weil divisor in Proposition \ref{prop:canonical_divisor} is Cartier. It is equivalent to show that the Weil divisor class 
\[\sum_{F \in \mathcal{H}(M_V)}-2 \, [\overline{Z}_{V,F}]\]
is an element of the subgroup $\CaCl(Y_V) \subseteq \Cl(Y_V)$, because the canonical divisor of $Y_V$ is defined up to linear equivalence. By Proposition \ref{prop:cartier_class_group}, this occurs if and only if there exists a set function $w: E \to \mathbb{Z}$ such that
\begin{equation}
\label{eq:gorenstein_1}
\sum_{e\in E} \left( w(e) \sum_{F \in \mathcal{H}_e(M_V)} [\overline{Z}_{V,F} ] \right)= \sum_{F \in \mathcal{H}(M_V)} -2 \,[\overline{Z}_{V,F}].
\end{equation}
After bringing the constant $w(e)$ inside the second sum, the left-hand side of Equation (\ref{eq:gorenstein_1}) becomes
\begin{equation}
\label{eq:gorenstein_2}
\sum_{e\in E}\sum_{F \in \mathcal{H}_e(M_V)} w(e)[\overline{Z}_{V,F} ].
\end{equation}
Furthermore, we can reindex the double sum in Expression (\ref{eq:gorenstein_2}) to get
\begin{equation}
\label{eq:gorenstein_3}
\sum_{F \in \mathcal{H}(M_V)} \sum_{e \in E \setminus F} w(e)[\overline{Z}_{V,F}].
\end{equation}
Then, Expression (\ref{eq:gorenstein_3}) equals the right-hand side of Equation (\ref{eq:gorenstein_1}) if and only if, for each $F \in \mathcal{H}(M_V)$,
\[
    \sum_{e \in E \setminus F} w(e) = -2.
\]
The complements of the corank 1 flats of $M_V$ are precisely the cocircuits of $M_V$, so we get the combinatorial condition stated in Theorem \ref{thm:gorenstein_condition} for when $Y_V$ is Gorenstein. A near-identical proof yields the condition for when $Y_V$ is $\mathbb{Q}$-Gorenstein.
\end{proof}

\appendix
\section{Proof of Lemma \ref{lem:chow_ring_computations}}
\begin{proof}[Proof of Lemma \ref{lem:chow_ring_computations}, (I)]
It follows from the definition of a nearly maximal compatible pair that $F_{s+1}$ must be a rank-$(s+1)$ flat, $F_{s+2}$ must be a rank-$(s+2)$ flat, etc. 

Choose $a,b \in E$ such that $a \in F_{s+1} \setminus \cl(I)$ and $b \in F_{s+2} \setminus F_{s+1}.$ (If $r= s+2,$ hence there is no $F_{s+2}$, then let $F_{s+2} = E.$) Due to the relations (\ref{eq:ideal_description_1}) in the Chow ring $A^*(W_V)$, we have the following equality:
\begin{equation}
\label{eq:proof_of_i_1}
y_b - y_a = \sum_{F \not\ni b}x_F - \sum_{F \not\ni a}x_F.
\end{equation}
If a proper flat $F$ contains neither $a$ nor $b$, then the element $x_F$ appears in both of the above sums. Therefore, we can rewrite Equation (\ref{eq:proof_of_i_1}) as
\begin{equation}
\label{eq:proof_of_i_2}
y_b - y_a = \sum_{F \ni a,\: F \not\ni b}x_F - \sum_{F \not\ni a,\: F \ni b}x_F.
\end{equation}
The flat $F_{s+1}$ does contain $a,$ but does not contain $b,$ so we can rewrite Equation (\ref{eq:proof_of_i_2}) as
\begin{equation*}
\label{eq:proof_of_i_3}
y_b - y_a  = \sum_{F \in S}x_F  + x_{F_{s+1}}- \sum_{F \not\ni a,\: F \ni b}x_F,
\end{equation*}
where $S$ is the collection of proper flats that contain $a$, do not contain $b,$ and are not equal to $F_{s+1}$.
Thus, we can write
\begin{equation}
\label{eq:proof_of_i_4}
x_{F_{s+1}} = y_b - y_a  - \sum_{F \in S}x_F  + \sum_{F \not\ni a,\: F \ni b}x_F.
\end{equation}
We multiply both sides of Equation (\ref{eq:proof_of_i_4}) by $x_{I, \mathscr{F}}$ and distribute to get
\begin{equation}
\label{eq:proof_of_i_5}
\begin{split}
& x_{F_{s+1}} \cdot x_{I , \mathscr{F}} = \\ & y_b \cdot x_{I , \mathscr{F}} - y_a \cdot x_{I , \mathscr{F}}  - \sum_{F \in S}x_F \cdot x_{I , \mathscr{F}}  + \sum_{F \not\ni a,\: F \ni b}x_F \cdot x_{I , \mathscr{F}}.
\end{split}
\end{equation}
We now deal with each term on the right-hand side of Equation (\ref{eq:proof_of_i_5}) one at a time. 

\noindent
\textit{First term.} The element $b\in E,$ by construction, is not a member of $F_{s+1}$. Thus, by the set of relations (\ref{eq:ideal_description_3}) in the ring $A^*(W_V)$, we have
\begin{equation*}
y_b \cdot x_{I, \mathscr{F}} = 0.
\end{equation*}

\noindent
\textit{Second term.} Let $I' = I \, \cup \, a$. Since $a \notin \cl(I)$, we have that $I'$ must be an independent set. Furthermore, since $I \subseteq F_{s+1}$ and $a \in F_{s+1}$, we have that $I' = I \cup a \subseteq F_{s+1}$. We conclude that $(I', \mathscr{F})$ is a maximal compatible pair of $M_V$. We have
\begin{equation*}
y_a \cdot x_{I, \mathscr{F}} = y_{i_1}y_{i_2}\cdots y_{i_s} y_{a} x_{F_{s+1}}x_{F_{s+2}}\cdots x_{F_{s+t}} = x_{I',\mathscr{F}}. 
\end{equation*}

\noindent
\textit{Third term.} Suppose $F$ is a member of the collection $S$. We will show that \[x_F \cdot x_{I , \mathscr{F}} = 0.\]
We will do cases based on the rank of $F$.
\begin{itemize}

\item If $\rk(F) < s$, then $F$ does not contain every element in $I$. So, by the relations (\ref{eq:ideal_description_3}), we are done.

\item If $\rk(F) = s$, then $F$ still doesn't contain every element in $I$. (If $F$ did, then $F = \cl(I)$. But $a\in F$ and $a \notin \cl(I)$.) So, by the relations (\ref{eq:ideal_description_3}), we are done.

\item If $\rk(F) = s+1$, we have that $F$ and $F_{s+1}$ are incomparable flats, as they are of the same rank but not equal. So, by the relations (\ref{eq:ideal_description_2}), we are done.

\item If $\rk(F) > s+1$, then there is a $F' \in \mathscr{F}$ with the same rank. However, they are not equal, as $b\in F_{s+2} \subseteq F'$, but $b \notin F$. So, $F$ and $F'$ are incomparable flats, and by the relations (\ref{eq:ideal_description_2}), we are done.
\end{itemize}
So we have
\begin{equation*}
\sum_{F \in S}x_F \cdot x_{I,\mathscr{F}} = 0.
\end{equation*}

\noindent
\textit{Fourth term.} Let $F$ be a proper flat that doesn't contain $a$ and contains $b$. We will show that $x_F \cdot x_{I, \mathscr{F}} = 0$. We do this by showing that $F$ and $F_{s+1}$ are incomparable flats. This can be seen by observing that $a \in F_{s+1} \setminus F$ and $b \in F \setminus F_{s+1}$. So, by the relations (\ref{eq:ideal_description_2}), we have
\begin{equation*}
\sum_{F \not\ni a,\: F \ni b}x_F \cdot x_{I,\mathscr{F}} = 0.
\end{equation*}

Putting the four terms together, we have
\begin{equation*}
x_{F_{s+1}} \cdot x_{I ,\mathscr{F}} = -x_{I' , \mathscr{F}}.
\end{equation*}
But $\deg_{W_V}(x_{I' , \mathscr{F}} \cap [W_V]) = 1,$ so we are done.
\end{proof}
\begin{proof}[Proof of Lemma \ref{lem:chow_ring_computations}, (II)]  
We follow a similar argument. Choose $a,b \in E$ such that $a \in F_{s+k} \setminus F_{s+k-1}$ and $b \in F_{s+k+1} \setminus F_{s+k}.$ (If $r=s+ k+1,$ hence there is no $F_{s+k+1}$, then let $F_{s+k+1} = E.$) Following a near-identical argument to the one above,
we get 
\begin{equation}
\label{eq:proof_of_ii_1}
\begin{split}
& x_{F_{s+k}} \cdot x_{I , \mathscr{F}} = \\ & y_b \cdot x_{I , \mathscr{F}} - y_a \cdot x_{I , \mathscr{F}}  - \sum_{F \in S}x_F \cdot x_{I , \mathscr{F}}  + \sum_{F \not\ni a,\: F \ni b}x_F \cdot x_{I , \mathscr{F}},
\end{split}
\end{equation}
where $S$ is the collection of proper flats that contain $a,$ do not contain $b,$ and are not equal to $F_{s+k}.$ We deal with each term on the right-hand side of Equation (\ref{eq:proof_of_ii_1}) one at a time.

\noindent
\textit{First term.} The element $b\in E,$ by construction, is not a member of $F_{s+k}$. Thus, by the set of relations (\ref{eq:ideal_description_3}) in the ring $A^*(W_V)$, we have
\begin{equation*}
y_b \cdot x_{I, \mathscr{F}} = 0.
\end{equation*}

\noindent
\textit{Second term.} The element $a\in E,$ by construction, is not a member of $F_{s+k-1}$. Thus, by the relations (\ref{eq:ideal_description_3}), we have
\begin{equation*}
y_a \cdot x_{I, \mathscr{F}} = 0. 
\end{equation*}

\noindent
\textit{Third term.} Suppose $F$ is a member of the collection $S$. We will show that $x_F \cdot x_{I , \mathscr{F}} = 0$. We will do cases based on the rank of $F$.
\begin{itemize}

\item If $\rk(F) < s$, then $F$ does not contain every element in $I$. So, by the relations (\ref{eq:ideal_description_3}), we are done.

\item If $\rk(F) = s$, then $F$ still doesn't contain every element in $I$. (If $F$ did, then $F = \cl(I)$. But $a\in F$ and $a \notin \cl(I)$.) So, by the relations (\ref{eq:ideal_description_3}), we are done.

\item If $s < \rk(F) < s+k,$ then there is a $F' \in \mathscr{F}$ with the same rank. However, they are not equal, as $a\in F,$ but $a \notin F'.$ So, $F$ and $F'$ are incomparable flats, and by the relations (\ref{eq:ideal_description_2}), we are done.

\item If $\rk(F) = s+k$, we have that $F$ and $F_{s+k}$ are incomparable flats, as they are of the same rank but not equal. So, by the relations (\ref{eq:ideal_description_2}), we are done.

\item If $\rk(F) > s+k$, then there is a $F' \in \mathscr{F}$ with the same rank. However, they are not equal, as $b\in F'$, but $b \notin F$. So, $F$ and $F'$ are incomparable flats, and by the relations (\ref{eq:ideal_description_2}), we are done.
\end{itemize}
So we have
\begin{equation*}
\sum_{F \in S}x_F \cdot x_{I,\mathscr{F}} = 0.
\end{equation*}

\noindent
\textit{Fourth term.} Let $F$ be a proper flat that doesn't contain $a$ and contains $b$. We will show that $x_F \cdot x_{I, \mathscr{F}} = 0$. We do this by showing that $F$ and $F_{s+k}$ are incomparable flats. This can be seen by observing that $a \in F_{s+k} \setminus F$ and $b \in F \setminus F_{s+k}$. So, by the relations (\ref{eq:ideal_description_2}), we have
\begin{equation*}
\sum_{F \not\ni a,\: F \ni b}x_F \cdot x_{I,\mathscr{F}} = 0.
\end{equation*}

Putting everything together, we have
\begin{equation*}
x_{F_{s+k}} \cdot x_{I ,\mathscr{F}} = 0,
\end{equation*}
so we are done.
\end{proof}
\begin{proof}[Proof of Lemma \ref{lem:chow_ring_computations}, (III)]
We start with 
\begin{equation}
\label{eq:proof_of_iii_1}
    \alpha \cdot x_{I, \mathscr{F}} = \left( \sum_{F \in \mathcal{L}_{\neq r} (M_V)} a(F)x_F \right) \cdot x_{I, \mathscr{F}} = \sum_{F \in \mathcal{L}_{\neq r} (M_V)} \left( a(F) x_F \cdot x_{I, \mathscr{F}} \right).
\end{equation}
We partition $\mathcal{L}_{\neq r}(M_V)$ into four sets:
\[\mathcal{L}_{\neq r}(M_V) = P_1 \sqcup P_2 \sqcup P_3 \sqcup P_4,\]
where
\begin{itemize}
    \item $P_1 = \mathscr{F}$
    \item $P_2 = \{\cl(I)\}$
    \item $P_3 = \{ F\in \mathcal{L}_{\neq r}(M_V)  : I \not \subseteq F\}$
    \item $P_4 = \{F\in \mathcal{L}_{\neq r}(M_V): I \subseteq F, F\notin \mathscr{F}, \rk(F) > s\}.$
\end{itemize}
It is straightforward to check that the sets are disjoint from each other. Here we justify why an arbitrary proper flat $F$ falls into one of these four sets. If $\rk(F) < s,$ then $I \not \subseteq F$, so $F \in P_3$. If $\rk(F) = s,$ then either $I \subseteq F$, in which case $F = \cl(I)$, or $I\not \subseteq F$, in which case $F\in P_3$. Lastly, if $\rk(F) > s$, then clearly $F$ is either in $P_1,$ $P_3,$ or $P_4$.
Using this partition, the right-hand side of Equation (\ref{eq:proof_of_iii_1}) becomes
\begin{equation}
\label{eq:proof_of_iii_2}
\begin{split}
& \sum_{F \in \mathscr{F}}a(F)x_F \cdot x_{I, \mathscr{F}} + a(\cl(I))x_{\cl(I)} \cdot  x_{I,\mathscr{F}} + \\ & \sum_{F \in P_3}a(F)x_F \cdot  x_{I, \mathscr{F}} + \sum_{F \in P_4} a(F)x_F \cdot x_{I, \mathscr{F}}.
\end{split}
\end{equation}
We deal with each term in Equation (\ref{eq:proof_of_iii_2}), one at a time.

\noindent
\textit{First term.} By Lemma \ref{lem:chow_ring_computations}, parts (I) and (II), we have that
\begin{equation}
\begin{split}
& \deg_{W_V}\left(\ \Big( \sum_{F \in \mathscr{F}}a(F)x_F \cdot x_{I, \mathscr{F}} \Big) \cap [W_V] \right) = \\ & \sum_{F \in \mathscr{F}} \Big( a(F) \deg_{W_V}\big((x_F \cdot x_{I, \mathscr{F}}) \cap [W_V]\big) \Big) = -a(F_{s+1}).
\end{split}
\end{equation}

\noindent
\textit{Second term.}
Let $\mathscr{F}'$ be the flag
\[\mathscr{F}':\cl(I) \subsetneq  F_{s+1} \subsetneq F_{s+2} \subsetneq \cdots \subsetneq F_{s+t}.\]
We observe that $(I, \mathscr{F}')$ is a maximal compatible pair. We have
\[x_{\cl(I)} \cdot x_{I, \mathscr{F}} = y_{i_1}y_{i_2}\cdots y_{i_s} x_{\cl(I)} x_{F_{s+1}}x_{F_{s+2}}\cdots x_{F_{s+t}} = x_{I,\mathscr{F}'}.\]
So,
\begin{equation}
\begin{split}
& \deg_{W_V}\Big((a(\cl(I))x_{\cl(I)} \cdot x_{I, \mathscr{F}}) \cap [W_V] \Big) = \\ & a(\cl(I)) \deg_{W_V}(x_{I, \mathscr{F}'} \cap [W_V]) = a(\cl(I)).
\end{split}
\end{equation}

\noindent
\textit{Third term.} By the relations (\ref{eq:ideal_description_3}) in $A^*(W_V)$, the third term is 0.

\noindent
\textit{Fourth term.} Let $F$ be an arbitrary proper flat in $P_4.$ Since $\rk(F) > s,$ there exists a flat $F' \in \mathscr{F}$ with the same rank as $F$. However, they cannot be the same flat; hence $F$ and $F'$ are incomparable. Thus, by the relations (\ref{eq:ideal_description_2}), the fourth term is 0.

Putting everything together, we see that we are done.
\end{proof}
\medskip
\printbibliography

@article{huh_wang,
    author = "June Huh and Botong Wang" ,
    title = "Enumeration of points, lines, planes, etc.",
    year = "2017",    
    journal = "Acta Mathematica",
    volume = "218",
    number = "2",
    pages = "297--317",
}

@article{stellahedral,
    author = "Christopher Eur and June Huh and Matt Larson",
    title = "Stellahedral geometry of matroids",
    year = "2023",
    journal = "Forum of Mathematics, Pi",
    volume = "11",
    pages = "e24",
}

@article{crowley_proudfoot,
author = "Colin Crowley and Nicholas Proudfoot",
title = "The geometry of zonotopal algebras II: Orlik–Terao algebras and Schubert varieties",
year = "2026",
journal = "Proceedings of the London Mathematical Society",
volume = "132",
number = "6",
pages = "e70168",
}

@inproceedings{llompart_descamps,
    author="F. Rossell{\'o} Llompart and S. Xamb{\'o} Descamps",
    title="Computing chow groups",
    year="1988",
    booktitle="Algebraic Geometry Sundance 1986",
    publisher="Springer Berlin Heidelberg",    
    pages="220--234",
}

@article{danilov,
    author = "V. I. Danilov",
     title = "The geometry of toric varieties",
     year = "1978",
   journal = "Russian Mathematical Surveys",
    volume = "33",
    number = "2",
     pages = "97--154",
}

@article{semismall,
author = "Tom Braden and June Huh and Jacob P. Matherne and Nicholas Proudfoot and Botong Wang",
title = "A semi-small decomposition of the Chow ring of a matroid",
year = "2022",
journal = "Advances in Mathematics",
volume = "409",
pages = "108646",
}

@article{woo_yong,
author = "Alexander Woo and Alexander Yong",
title = "When is a Schubert variety Gorenstein?",
year = "2006",
journal = "Advances in Mathematics",
volume = "207",
number = "1",
pages = "205--220",
}

@article{ardila_boocher,
    author="Federico Ardila and Adam Boocher",
      title="The closure of a linear space in a product of lines",      
      year="2015",
      journal = "Journal of Algebraic Combinatorics",
      volume = "43",
      pages = "199--235",     
}

@article{proudfoot_2018, 
author="Nicholas Proudfoot",
title="The algebraic geometry of Kazhdan–Lusztig–Stanley polynomials",
year="2018", 
journal="EMS Surveys in Mathematical Sciences",
volume="5",  
pages={99--127},
}

@article{fulton_sturmfels,
author = "William Fulton and Bernd Sturmfels",
title = "Intersection theory on toric varieties",
year = "1997",
journal = "Topology",
volume = "36",
number = "2",
pages = "335-353",
}

@article{operational_chow_cohomology,
    author    = "William Fulton and Robert MacPherson",
    title     = "Categorical framework for the study of singular spaces",
    year      = "1981",
    journal   = "Memoirs of the American Mathematical Society",
    volume   = "31",
    number   = "243",
}

@article{singular_hodge_theory,
       author="Tom Braden and June Huh and Jacob P. Matherne and Nicholas Proudfoot and Botong Wang",
        title = "Singular Hodge theory for combinatorial geometries",
        year = 2026,
      journal = "arXiv e-prints",    
        pages = "arXiv:2010.06088",          
}

@book{intersection_theory,
author="William Fulton",
title="Intersection Theory",
year="1998",
publisher="Springer New York",
}

@phdthesis{singular_thesis,
	author		= "Yiyu Wang",
	title		= "Microlocal multiplicities of matroid Schubert varieties",
    year		= "2025",
	school		= "University of Wisconsin-Madison",
}
\end{document}